\documentclass[12pt,a4]{article}

\usepackage{amsmath,amssymb,amsthm}

\topmargin= - 7 true mm
\newtheorem{theorem}{Theorem}

\newtheorem{coro}{Corollary}
\newtheorem{remark}{Remark}
\newcommand{\E}{{\bf E\hskip 0.3 mm}}
\renewcommand{\P}{{\bf P\hskip 0.3 mm}}

\title{New and refined bounds for expected maxima of fractional Brownian motion}

\author{Konstantin Borovkov$^1$, Yuliya Mishura$^2$,\\ Alexander Novikov$^3$
and Mikhail Zhitlukhin$^4$}

\date{}

\begin{document}
\maketitle

\footnotetext[1]{School of Mathematics and Statistics, The University of Melbourne, Parkville 3010, Australia; e-mail: borovkov@unimelb.edu.au.}

\footnotetext[2]{Mechanics and Mathematics
Faculty, Taras Shevchenko National University of Kyiv, Volodymyrska str.~64, 01601 Kyiv, Ukraine; email: myus@univ.kiev.ua.}

\footnotetext[3]{School of Mathematical and Physical Sciences, University of Technology Sydney, PO Box 123, Broadway,
Sydney, NSW 2007,   Australia; email: Alex.Novikov@uts.edu.au.}

\footnotetext[4]{Steklov Mathematical Institute of Russian Academy of Sciences, Gubkina str.~8,
119991, Moscow, Russia; email: mikhailzh@mi.ras.ru.}

\begin{abstract}
For the fractional Brownian motion $B^H$ with the Hurst parameter value $H$ in (0,1/2), we derive new upper and lower bounds for the difference between the expectations of the maximum of $B^H$ over [0,1] and  the maximum of $B^H$ over the discrete set of values $ in^{-1},$ $i=1,\ldots, n.$ We use these results to improve our earlier upper bounds for the expectation of the maximum of $B^H$ over $[0,1]$ and derive new upper bounds for Pickands' constant.


\end{abstract}

\section{Introduction}

Let $B^H=(B_t^H)_{t\ge 0}$ be a fractional Brownian motion (fBm) process with Hurst
parameter $H\in(0,1)$, i.e.\ a zero-mean continuous Gaussian process with the
covariance function
$\E  B_s^H B_t^H  = \frac12 (s^{2H} + t^{2H} - |s-t|^{2H}), $ $ s,t \ge 0.$
Equivalently, the last condition can be stated as $ B^H_0=0$ and
\begin{equation}
 \E ( B_s^H - B_t^H)^2= |s-t| ^{2H},\quad s,t \ge 0 .
 \label{L2}
\end{equation}
Recall that the Hurst parameter $H$ characterizes the type   of the  dependence of the increments of the fBm. For $H\in (0,\frac12)$ and $H\in (\frac12, 1)$, the increments of $B^{H}$ are  respectively  negatively and positively correlated, whereas the process $B^{1/2}$ is the standard Brownian motion which  has independent increments. The fBm processes are important construction blocks in various application areas, the ones with $H>\frac12$ being of interest as their increments exhibit long-range dependence, while it was shown recently that fBm's with $H<\frac12$   can  be well fitted to real life telecommunications, financial markets with stochastic volatility and other financial data (see, e.g., \cite{ArGl,BaFrGa}). For detailed exposition of the theory of   fBm processes, we refer the reader to \cite{BiHuOkZh, Mis, Nou} and references therein.

Computing the value of the expected maximum
 \[
 M^H:=\E \max_{0\le t \le 1} B_t^H
 \]
is an important question arising in a number of applied problems, such as finding the likely magnitude of the strongest earthquake to occur this century in a given region  or the speed of the strongest wind gust a tall building has to withstand during its lifetime etc. For the standard Brownian motion $B^{1/2}$, the exact value
of the expected maximum is $\sqrt{\pi/2}$, whereas for all other $H\in (0,1)$ no closed-form
expressions for the expectation  are  known. In the absence of such results, one standard approach to computing $M^H$ is to evaluate instead its approximation
$$
 M^H_n:=\E \max_{1\le i \le n} B_{i/n}^H, \quad  n\ge 1,
$$
(which can, for instance,  be done using simulations) together with the approximation error
\[
\Delta_n^H :=  M^H - M^H_n.
\]

Some bounds for $\Delta_n^H$ were recently established  in~\cite{BMNZ}.
The main result of the present note is an improvement of the following upper bound for $\Delta_n^H$ obtained in Theorem~3.1 of   ~\cite{BMNZ}: for $n\ge 2^{1/H},$
\begin{equation}\label{old}
\Delta_n^H \le \frac{2 (\ln n)^{1/2}}{n^{H}}\biggl(1+\frac4{n^H}+\frac{0.0074}{(\ln n)^{3/2}}\biggr).
\end{equation}
Lower bound for $\Delta_n^H$ is obtained as well and we study for which $H$ and $n$ upper and lower bounds  hold simultaneously. We also obtain a new upper bound for the expected maximum
$M^H$ itself and some functions of it, which refines previously known results (see e.g.\ \cite{BMNZ,S}), and
use it to derive an improved upper bound for the so-called Pickands' constant, which is the basic constant in the extreme value theory of Gaussian
processes.

The paper is organized as follows: Section \ref{main} contains the results, with comments and examples, and Section \ref{proofs} contains the proofs.

\section{Main results}\label{main} From now on, we always    assume that $H\in(0,\frac12)$. The next theorem is the main result of the note. As usual,   $\lfloor x\rfloor $
and $\lceil x\rceil$ denote\ the floor and the ceiling of the real number~$x$.

\begin{theorem}\label{Thm1}
1) For any $\alpha>0$ and $n\ge 2^{1/\alpha} \vee
(1+\frac{\alpha}{1+\alpha})^{1/(2\alpha H)}$  one has
\begin{equation}
\frac{\Delta_n^H}{n^{-H}(\ln n)^{1/2}}
\le
\frac{(1-\lfloor n^\alpha \rfloor^{-1})^H (1+\alpha)^{1/2}}{1-\lfloor n^\alpha \rfloor^{-H}( 1+ \alpha/(1+\alpha))^{1/2}}.
\label{2}
\end{equation}

2) For any $n\ge 2$ one has
\begin{equation}\label{2+}
\frac{\Delta_n^H}{n^{-H}(\ln n)^{1/2}} \ge n^{H}\biggl(
\frac{L}{(\ln n^{H})^{1/2}} - 1\biggr)^+,\end{equation}
where $ L=1/\sqrt{4 \pi e \ln 2} \approx 0.2$ and $a^+=a\vee 0.$

\end{theorem}

\begin{remark} \rm
Note that inequality~\eqref{2+}  actually holds  for all $H\in (0,1).$
\end{remark}
\begin{remark} \rm
Let us study for which $H$ and $n$ upper and lower bounds \eqref{2} and \eqref{2+} hold simultaneously under assumption that \eqref{2+} is non-trivial. For non-triviality we need to have $n<\exp{\frac{L^2}{H}}.$ In order to have $2^{1/\alpha}\leq \exp{\frac{L^2}{H}}$ we restrict $\alpha$ to $\alpha\geq \frac{H\ln 2}{L^2}.$ In order to have
$(1+\frac{\alpha}{1+\alpha})^{1/(2\alpha H)}\leq \exp\{\frac{L^2}{H}\}$, or, what is equivalent, \begin{equation}\label{alpha}\left(1+\frac{\alpha}{1+\alpha}\right)^{1/ \alpha }\leq \exp\{ 2L^2 \},\end{equation} we note that the function $q(\alpha)=(1+\frac{\alpha}{1+\alpha})^{1/ \alpha }$ continuously strictly decreases in $\alpha\in (0,\infty)$ from $e$ to 1, and taking into account the value of $L$, we get that there is a unique root $\alpha^*\approx 7.48704$ of the equation $\left(1+\frac{\alpha}{1+\alpha}\right)^{1/ \alpha }= \exp\{ 2L^2 \}$ and for $\alpha\geq\alpha^*$ we have that \eqref{alpha} holds. Therefore for $\alpha>\alpha^*$, $H<\frac{\alpha^*L^2}{\ln 2}\approx 0.456$ and $\exp\{\frac{L^2}{H}\}>n>2^{1/\alpha} \vee
(1+\frac{\alpha}{1+\alpha})^{1/(2\alpha H)}$ we have that lower bound \eqref{2+} holds and is non-trivial.
Moreover, $2^{1/\alpha}<2^{1/\alpha^*}(\approx 1.097)<\exp\{\frac{L^2}{H}\}$, $(1+\frac{\alpha}{1+\alpha})^{1/(2\alpha H)}<(1+\frac{\alpha^*}{1+\alpha^*})^{1/(2\alpha^* H)}=\exp\{\frac{L^2}{H}\}$, so the interval $(2^{1/\alpha} \vee
(1+\frac{\alpha}{1+\alpha})^{1/(2\alpha H)}, \;\exp\{\frac{L^2}{H}\})$ is non-empty and for such $n$ upper bound \eqref{2} holds. The only question is if this interval contains the integers. If it is not the case we can increase the value of $\alpha$. For example, put $H=0.01$, $\alpha=16$, then it holds that the interval $(2^{1/\alpha} \vee
(1+\frac{\alpha}{1+\alpha})^{1/(2\alpha H)},\; \exp\{\frac{L^2}{H}\})=(1.044\vee 7.534,\; 20.085)=(  7.534,\; 20.085).$
\end{remark}
\begin{remark} \rm
Consider the sequence $\alpha=\alpha (m)\to 0$ slowly enough as $m\to\infty$ (take, e.g., $\alpha(m)=(\ln\ln m)/\ln m$). Then for sufficiently large enough  $m$ we have that $m\ge 2^{1/\alpha(m)} \vee
(1+\frac{\alpha(m)}{1+\alpha(m)})^{1/(2\alpha(m) H)}$ therefore for such $m$ the upper bound ~\eqref{2} holds. Returning to standard notation $n$ for  the argument, we obtain from the   upper bound in~\eqref{2} that, for any fixed $H\in (0,\frac12),$ one has
\begin{equation}
\Delta_n^H \le
n^{-H}{(\ln n)^{1/2}} (1+o(1)), \quad n\to\infty,
\label{r1}
\end{equation}
which refines \eqref{old}.\end{remark}

\begin{remark} \rm
Recall that, in the case  of the standard Brownian motion ($H=\frac12$),
the exact asymptotics of  $\Delta_{n}^{1/2}$ are well-known:
\[
\Delta_{n}^{1/2}= n^{-1/2} (\beta +o(1)),\quad  n\to \infty,
\]%
where $\beta =-\zeta(1/2)/\sqrt{2\pi }=0.5826\ldots$ and  $\zeta(\cdot) $ is
the  Riemann zeta function (see~\cite{Siegmund}). Comparing it with
\eqref{r1},  we see that now we have additional logarithmic multiplier.
 \end{remark}

The next simple assertion enables one to use the upper bound obtained in Theorem~\ref{Thm1} to get an upper bound for the approximation rate of the expectation of a function of the maximum of an fBm. Such a result is required, for instance, for bounding convergence rates when approximating Bayesian estimators in irregular statistical experiments (see, e.g.,~\cite{KKNL,NKL}).

Set
\[
\overline{B}^H_1 :=  \max_{0\le t\le 1} B_t^H,
\quad
\overline{B}^H_{n,n}: =  \max_{1\le i\le n} B_{i/n}^H,
\quad
\Delta_n^{H,f} := \E f(\overline{B}^H_1) - \E f(\overline{B}^H_{n,n})
\]
and, for a function
$f:\mathbb{R}\to \mathbb{R}$, denote its continuity modulus by
\[
\omega_{\delta, h}(f):=\sup_{0\le s< t\le (s+\delta)\wedge h }|f(s)-f(t)|,\quad h, \delta >0.
\]

\begin{coro}
\label{Thm3}
Let $f\ge 0$ be an arbitrary non-decreasing function on $\mathbb{R}$ such that
$f(x) = o\bigl(\exp ( (x- M^H)^2/2)\big)$ as $x\to \infty$. Then, for  any number $M>M^H,$
\[
\Delta_n^{H,f} \le
\omega_{\Delta_n^H,M} (f) +
\int_M^\infty f(x) (x-M^H) \exp \bigl\{- (x- M^H)^2/2\bigr\}dx.
\]
\end{coro}

To roughly balance the contributions from the two terms in the bound, one may wish to choose~$M$ so that
$\exp \bigl\{- (M- M^H)^2/2\bigr\}$ would be of the same order of magnitude as~$\Delta_n^H$ (as for regular functions~$f$  that are mostly of interest in applications are locally Lipschitz, so that  $\omega_{\delta, h}(f)$ admits a linear upper bound in~$\delta$). To that end, one can take
$M: = M^H +  (-2\ln \Delta_n^H)^{1/2} +\mbox{const}$ (assuming that $n$ is large enough so that $\Delta_n^H<1$). We will illustrate that in   two special cases where   $f$ is the exponential function (this case corresponds to the above-mentioned applications from~\cite{KKNL,NKL}) and a power
function, respectively.

\medskip
\noindent{\bf Example~1.}
Assume that $f(x) = e^{ax}$ with  a fixed $a>0$,  and that $\Delta_n^H<1$. Choosing $M:=M^H+a +|2\ln \Delta_n^H|^{1/2}$ we  get
\[
\omega_{\Delta_n^H,M} (f)\le e^{aM} \Delta_n^H = \exp\{a M^H+a^2 +a|2\ln \Delta_n^H|^{1/2} \}\Delta_n^H
\]
and, setting $y:=x-M^H$ and using the well-known bound for the Mills' ratio for the normal distribution, obtain that
\begin{align*}
\int_M^\infty f(x) (x& -M^H)    \exp \bigl\{- (x- M^H)^2/2\bigr\}dx
 = e^{aM^H} \int_{M-M_H}^\infty y  e^{-y^2/2+ay}dy
 \\
 &
 = e^{aM^H+a^2/2}
 \biggl[ \int_{M-M_H}^\infty (y-a) e^{-(y-a)^2/2}dy + a \int_{M-M_H}^\infty   e^{-(y-a)^2/2}dy\biggr]
 \\
 &
 \le e^{aM^H+a^2/2} \biggl(1+\frac{a}{M-M^H-a} \biggr) e^{-(M-M^H-a)^2/2}
 \\
 & =e^{aM^H+a^2/2} \biggl(1+\frac{a}{ |2 \ln \Delta_n^H|^{1/2}} \biggr) \Delta_n^H.
\end{align*}
Therefore
\[
\Delta_n^{H,f}
\le
  e^{aM^H+a^2/2} \biggl(1+
e^{ a^2/2 +a|2\ln \Delta_n^H|^{1/2}}+ \frac{a}{|2 \ln \Delta_n^H|^{1/2}}
\biggr)\Delta_n^H   .
\]

\medskip
\noindent{\bf Example~2.}
For the function $f(x) = x^p$,  $p\ge 1,$ one clearly has
\[
\Delta_n^{H,f}
 \le pM^{p-1}\Delta_n^H
 + \int_M^\infty x^p (x- M^H) \exp \bigl\{- (x- M^H)^2/2\bigr\}dx.
\]
Observe that $x^p=(x-M^H)^p \Bigl(1+\frac{M^H}{x-M^H}\Bigr)^p\le (x-M^H)^p \Bigl( \frac{M }{M-M^H}\Bigr)^p$ for $  x\ge M,$ while, for any $A>0,$
\begin{align*}
\int_{A}^\infty  z^{p+1 }  e^{- z^2/2}dz= -\int_{A}^\infty  z^{p } d  e^{- z^2/2} =
  A^p e^{-A^2/2} + p\int_{A}^\infty  z^{p-1 }  e^{- z^2/2}dz,
 \end{align*}
where the last integral does not exceed $A^{-2}\int_{A}^\infty  z^{p+1 }  e^{- z^2/2}dz,$ so that \[
\int_{A}^\infty  z^{p+1 }  e^{- z^2/2}dz \le \frac{ A^p e^{-A^2/2} }{1-pA^{-2}} \quad\mbox{for \ $A^2>p$.}
\]
Hence, choosing
$A:=M-M^H= |2 \ln \Delta_n^H|^{1/2}$, we obtain that, for $\Delta_n^H <e^{-p/2},$
\[
\Delta_n^{H,f}
\le \bigl(M^H + |2 \ln \Delta_n^H|^{1/2}\bigr)^{p-1}
\biggl(p
 + \frac{  M^H + |2 \ln \Delta_n^H|^{ 1/2}}{1-p |2 \ln \Delta_n^H|^{-1}}\biggr)\Delta_n^H.
\]

\medskip

Finally, in the next corollary we use Theorem~\ref{Thm1}   to improve the   known upper bound $M^H< 16.3 H^{-1/2}$ for the expected maximum  $M^H$
from Theorem~2.1(ii) in~\cite{BMNZ}.


\begin{coro}
\label{Thm4}
Assume that  $H$ is such that $2^{2/H}$ is integer. Then
\[
M^H <  1.695 H^{-1/2}. 
\]
\end{coro}

\begin{remark}  \rm
If $2^{2/H}$ is not integer then, in the above formula, one
can use instead of $H$ the largest value $\widetilde H < H$ such that $2^{2/\widetilde H}$ is integer, i.e.\ $\widetilde H = 2/ \log_2 \lceil 2^{2/H} \rceil$. This is so since it follows from Sudakov--Fernique's inequality (see e.g.\ Proposition~1.1 and Section~4 in~\cite{BMNZ}) that the expected maximum $M^H$  is a non-increasing function of~$H$.
\end{remark}


\begin{remark} \rm
Our new upper bound for $M^H$ can be used to improve Shao's upper bound
from~\cite{S} for Pickands' constant $\mathcal{H}_H$, which is a basic
constant in the extreme value theory of Gaussian processes and is of
interest in a number of applied problems. That constant appears in the
asymptotic representation for the tail probability of the maxima of
stationary Gaussian processes in the following way (see e.g. \cite{Pickands}).

Assume that
$(X_t)_{t\ge 0}$ is a stationary Gaussian process with zero mean and unit
variance of which the covariance function $r(v):=\E X_t X_{t+v},$ satisfies
the following relation: for some $C>0$ and $H \in(0,1]$, one has
$r(t) = 1 - C|t|^{2H} + o(|t|^{2H} )$ as $t\to 0$. Then, for each fixed $T>0$ such
that $\sup_{\varepsilon\le t\le T} r(t) < 1$ for all $\varepsilon>0,$
\[
\P\Bigl(\sup_{0\le t\le T} X_t > u\Bigr) =  C^{1/(2H) } \mathcal{H}_H
(2\pi)^{-1/2}e^{- {u^2}/2}  u^{  1/H -1} (T+o(1)), \quad u\to\infty.
\]
It was shown in~\cite{S}  that, for $H \in(0,1/2],$
\[
\mathcal{H}_H  \le \left(2^{1/2} {e H }  M^{H}\right)^{1/H }.
\]
Using our Corollary~\ref{Thm4}, we obtain the following new upper bound for Pickands' constant:
\[
\mathcal{H}_H  < (42.46  H )^{1/(2H) }, \quad H \in (0,1/2],
\]
which is superior to Shao's bound
\[
\mathcal{H}_H  \le \left\{1.54H+ 4.82 H^{1/2} (4.4 - H  \ln (0.4 +1.25 /H ))^{1/2}\right\} ^{1/H }, \quad H \in (0,1/2]
\]
(see (1.5) in~\cite{S}; there the notation $a := 2H$ is used). Fore example, the ratio of our bound to Shao's equals 0.344 when $H=0.45$ and is 0.046 when $H=0.15$.
\end{remark}

\section{Proofs}\label{proofs}

\begin{proof}[Proof of Theorem~\ref{Thm1}]
First we will prove   \eqref{2}.
Let $n_k : = n m^k$, $k\ge 0$, where we set  $m:= \lfloor n^\alpha\rfloor \ge
2$.  It follows from the continuity of $B^H$ and monotone convergence theorem that
\begin{equation}
\Delta_n^H =
\sum_{k=0}^\infty (M^H_{n_{k+1}} -M^H_{n_k}).
\label{3}
\end{equation}
Although this step is common with the proof of Theorem~3.1 in~\cite{BMNZ}, the rest of the argument uses a different idea. Namely, we apply  Chatterjee's inequality (\cite{C}; see also Theorem~2.2.5 in~\cite{AdTa})   which, in its  general formulation,
states the following. For any   $N$-dimensional Gaussian random vectors $X=(X_1,\ldots, X_N)$, $Y=(Y_1, \ldots, Y_N)$ with common means:
$\E X_i = \E Y_i$ for $1\le i\le  N$, one has
\begin{equation}
\bigl|\E \max_{1\le i\le N} X_i - \E \max_{1\le i\le N} Y_i\bigr| \le  (\gamma \ln N)^{1/2},\quad
\gamma := \max_{1\le i<j\le N} |d_{ij}(X) - d_{ij}(Y)|,
\label{Chat}
\end{equation}
where, for a random vector $Z\in {\mathbb R}^N,$
we set $d_{ij} (Z) := \E(Z_i-Z_j)^2$, $1\le i,j\le N.$

To be able to apply  inequality \eqref{Chat}  to the terms in the sum on the right-hand side of~\eqref{3}, for each $k\ge 0$ we  introduce   auxiliary vectors $X^k,Y^k\in \mathbb R^{n_{k+1}}$ by letting
\[
X^k_i: = B^H_{i/n_{k+1}},\quad Y^{k }_i :=B^H_{\lceil i/m \rceil/n_k},\quad 1\le i\le n_{k+1}.
\]
Note that $
M^H_{n_{k+1}}=\E\max_{1\le i\le n_{k+1}} X^{k}_i $ and $M^H_{n_{k}}= \E\max_{1\le i\le n_{k+1}} Y_i^k,$ so that now~\eqref{Chat} is applicable. Next we will show  that
\[
\gamma^k := \max_{1\le i<j\le n_{k+1}}
|d_{ij}(X^{k}) - d_{ij} (Y^{k})|
\le  {n_k^{-2H}}(1-m^{-1})^{2H}.
\]
Indeed, one can clearly write down the representations  $i=a_im + b_i$, $j=a_jm + b_j$
with integer $a_j\ge  a_i\ge 0$ and   $1\le b_i,b_j
\le m$, such that $b_j>b_i$ when $a_i=a_j$. Then it follows from~\eqref{L2} that
\[
d_{ij}(X^{k }) = \left(\frac{(a_j-a_i)m +
b_j-b_i}{n_{k+1}}\right)^{2H},
\qquad
d_{ij}(Y^{k }) = \left(\frac{(a_j-a_i)m}{n_{k+1}}\right)^{2H}.
\]
Since for $2H\le 1$ the function $x\mapsto x^{2H},$ $x\ge 0,$ is concave, it is also sub-additive, so that $x^{2H} - y^{2H} \le
(x-y)^{2H}$ for any $x\ge y\ge 0$. Setting  $x:=d_{ij}(X^{k})\vee d_{ij}(Y^{k})$ and $y:=d_{ij}(X^{k})\wedge d_{ij}(Y^{k})$, this yields the desired bound
\[
|d_{ij}(X^{k}) -d_{ij}(Y^{k})| \le
\left(\frac{|b_i-b_j|}{n_{k+1}}\right)^{2H} \le
\left(\frac{m-1}{n_{k+1}}\right)^{2H} = \frac{1}{n_k^{2H}} \left(1-\frac1m\right)^{2H}.
\]

Now it follows from~\eqref{Chat} that
\begin{align*}
M^H_{n_{k+1}}   -M^H_{n_{k }} &  \equiv \E \max_{1\le i\le n_{k+1}} X^{k}_i -  \E \max_{1\le i\le n_{k+1}} Y^{k}_i
\\
&
\le (\gamma^k \ln n_{k+1})^{1/2}
\le
\frac{(1-m^{-1})^H}{n^H m^{kH}}(\ln n + (k+1)\ln m)^{1/2}
\\
& \le
 \frac{(\ln n)^{1/2}}{n^H}(1-m^{-1})^H \frac{(1+ \alpha + \alpha k)^{1/2}}{m^{kH}}.
\end{align*}
   The last bound  together with~\eqref{3}  leads to
\[
\Delta_n^H \le  \frac{(\ln n)^{1/2}}{n^H}(1-m^{-1})^H
 \sum_{k=0}^\infty\frac{(1+ \alpha + \alpha k)^{1/2}}{m^{kH}}.
\]
The sum of the series on the right hand side is exactly  $\alpha^{1/2}\Phi (m^{-H},-\frac12, 1+\alpha^{-1})$, where $\Phi$ is the Lerch transcendent function. For our purposes, however, it will be convenient just to use the elementary bound  $(1+\alpha+\alpha k)^{1/2} \le (1+\alpha)^{1/2}
( 1+ \alpha/(1+\alpha))^{k/2},$ to get
\[
\Delta_n^H \le  \frac{(\ln n)^{1/2}}{n^H}\cdot\frac{(1-m^{-1})^H (1+\alpha)^{1/2}}{1-m^{-H}( 1+ \alpha/(1+\alpha))^{1/2}}.
\]

The right inequality in \eqref{2} is proved. To establish the left one, note that, on the one hand, it was shown in Theorem~2.1 \cite{BMNZ}  that $M^H \ge L H^{-1/2}$ for all $H\in (0,1)$.

On the other hand, it follows from Sudakov--Fernique's inequality (see e.g.\ Proposition~1.1 in~\cite{BMNZ}) that, for any fixed~$n\ge 1$, the quantity $M_n^H$ is non-increasing in $H$, and it follows from Lemma~4.1 in~\cite{BMNZ} that
\[
M^0_n:=\lim_{H\to 0} M^H_n=2^{-1/2} \E \overline{\xi}_n, \qquad \overline{\xi}_n:=  \max_{1\le i\le n} \xi_i,
\]
where $\xi_i$ are i.i.d.\ $N(0,1)$-distributed random variables. Furthermore, the last expectation admits the following upper bound:
\begin{align}
\E \overline{\xi}_n \le \sqrt{2\ln n},\quad n\ge 1.
\label{Enorm}
\end{align}
Although that bound has been known for some time, we could not find a suitable literature reference or stable Internet  link for it. So we decided to include a short   proof thereof for completeness' sake.
By Jensen's inequality, for any $s\in \mathbb{R},$
\begin{align*}
e^{s  \E \overline{\xi}_n}
 \le \E e^{s\overline{\xi}_n}
 =\E \max_{1\le i\le n} e^{s\xi_i}
\le \E \sum_{1\le i\le n}  e^{s\xi_i}= \sum_{1\le i\le n} \E e^{s\xi_i}
 =n e^{s^2/2},
\end{align*}
so that $\E \overline{\xi}_n\le s^{-1}\ln n +s/2.$ Minimizing in~$s$ the expression on the right-hand side   yields the desired bound~\eqref{Enorm}.

From the above results, 
we obtain that
\begin{align*}
M^H - M_n^H
& \ge
M^H - M_n^0 \ge L H^{-1/2} - (\ln n)^{1/2}
\\
& = n^{-H}(\ln n)^{1/2} \cdot  n^{H} \bigl(
 {L}{(H\ln n)^{-1/2}} - 1\bigr),
\end{align*}
which  completes the proof of Theorem~\ref{Thm1}.
\end{proof}

\begin{proof}[Proof of Corollary~\ref{Thm3}]
Since $f\ge 0,$ for any $M>M^H$ we have
\begin{align}
\Delta_n^{H,f} &\le \E \bigl(f(\overline{B}^H_1)-f(\overline{B}^H_{n,n}); \, \overline{B}^H_1 \le M\bigr) +
\E \bigl(f(\overline{B}^H_1);\, \overline{B}^H_1 > M\bigr)
\notag
\\
&\le \omega_{ \Delta_n^H,M} (f) +
\int_M^\infty f(x) dF(x),
\label{4}
\end{align}
where $F(x):=\P (\overline{B}^H_1\le x)$. From the well-known Borell--TIS inequality for Gaussian processes (see,
e.g., Theorem~2.1.1 in~\cite{AdTa}) it follows that, for any $u>0,$
\[
\P\bigl(\overline{B}^H_1 - M^H > u\bigr) \le e^{- {u^2}/{2}}.
\]
Therefore, for any $M>M^H$, integrating by parts, using the assumed property that $f(x) \exp \bigl\{- (x- M^H)^2/2\bigr\}\to 0$ as $x\to \infty$, and then again integrating by parts, we can write
\begin{align*}
\int_M^\infty  & f(x) \,dF(x) = f(M)(1-F(M))
 + \int_M^\infty (1-F (x))\, df(x)
 \\&
 \le f(M) \exp \bigl\{- (M- M^H)^2/2\bigr\} +
\int_M^\infty \exp \bigl\{- (x- M^H)^2/2\bigr\} df(x)
\\ &
 = -\int_M^\infty f(x) d\exp \bigl\{- (x- M^H)^2/2\bigr\}.
\end{align*}
Together with~\eqref{4} this establishes the assertion of Corollary~\ref{Thm3}.
\end{proof}

\begin{proof}[Proof of Corollary~\ref{Thm4}]
Using Chatterjee's inequality~\eqref{Chat} with the zero vector $Y$, we get for any $n\ge 1$
the bound $M^H_n \le ((1-n^{-2H})\ln n)^{1/2}$, so that we obtain from Theorem~\ref{Thm1} that
\begin{align*}
M^H
& \le \Delta_n^H +((1-n^{-2H})\ln n)^{1/2}
\\
& < \biggl[\frac{  n^{-H}(1+\alpha)^{1/2}}{1-m^{-H}( 1+ \alpha/(1+\alpha))^{1/2}}
 + (1-n^{-2H}) ^{1/2}\biggr](\ln n)^{1/2} .
\end{align*}
Now choosing $n:=4^{1/H}$ (which was assumed to be integer) and $\alpha:=2$, we get $m=n^\alpha=4^{2/H}$ and
\begin{align*}
M^H < H^{-1/2}\biggl[\frac{  4^{-1} 3^{1/2}}{1-16^{-1} (5/3)^{1/2}}
 + (1-16^{-1}) ^{1/2}\biggr](\ln 4)^{1/2} <1.695 H^{-1/2}.
\end{align*}
\end{proof}

{\bf Acknowledgements.}  This research was supported by the ARC  Discovery grant DP150102758.  The work of  M.~Zhitlukhin was supported by the Russian Science Foundation project 14--21--00162.

\end{document}